\documentclass[reqno, 11pt]{amsart}
\usepackage{mathtools}
\usepackage{amsmath}
\usepackage{amssymb}
\usepackage{yhmath}
\usepackage{graphicx}
\usepackage{ mathrsfs }
\usepackage{bbm}
\usepackage{xcolor}
\usepackage{tikz-cd}
\usepackage{tikz}
\usetikzlibrary{patterns}
\usepackage{hyperref}

\DeclareMathAlphabet{\mathpzc}{OT1}{pzc}{m}{it}

\usepackage{thmtools}
\usepackage{thm-restate}

\usepackage{caption}

\newtheorem{theorem}{Theorem}[section]

\newtheorem*{claim*}{Claim}

\newtheorem{lemma}[theorem]{Lemma}

\newtheorem{proposition}[theorem]{Proposition}

\newtheorem{prop}[theorem]{Proposition}

\theoremstyle{definition}

\newtheorem{Def}[theorem]{Definition}

\theoremstyle{remark}

\newtheorem{Rmk}[theorem]{Remark}

\numberwithin{equation}{section}

\newcommand{\op}{\operatorname}

\newcommand{\be}{\begin{equation}}
\newcommand{\ee}{\end{equation}}
\newcommand{\Ga}{\Gamma}

\newcommand{\N}{\mathbb N}
\newcommand{\ga}{\gamma}

\newcommand{\La}{\Lambda}

\newcommand{\cal}{\mathcal}

\newcommand{\SO}{\op{SO}}

\newcommand{\F}{\cal F}

\newcommand{\G}{\Gamma}

\renewcommand{\frak}{\mathfrak}

\newcommand{\fa}{\mathfrak a}

\renewcommand{\i}{\op{i}}

\newcommand{\Lie}{\op{Lie}}

\begin{document}

\title[Non-concentration property of Patterson-Sullivan measures]{Non-concentration property of Patterson-Sullivan measures for\\Anosov subgroups}

\author{Dongryul M. Kim}
\address{Department of Mathematics, Yale University, New Haven, CT 06511}
\email{dongryul.kim@yale.edu}

\author{Hee Oh}
\address{Department of Mathematics, Yale University, New Haven, CT 06511}
\email{hee.oh@yale.edu}
\thanks{
 Oh is partially supported by the NSF grant No. DMS-1900101.}
\begin{abstract}
Let $G$ be a connected semisimple real algebraic group.  For a Zariski dense Anosov subgroup $\Gamma<G$ with respect to a parabolic subgroup $P_\theta$, we prove that any $\Gamma$-Patterson-Sullivan  measure charges no mass on any proper subvariety of $G/P_\theta$. 
More generally, we prove that for a Zariski dense $\theta$-transverse subgroup $\Gamma<G$, any $(\Gamma, \psi)$-Patterson-Sullivan measure charges no mass on any proper subvariety of $G/P_\theta$, provided the $\psi$-Poincar\'e series of $\Gamma$ diverges at $s=1$. In particular, our result also applies to relatively Anosov subgroups.
\end{abstract}

\maketitle

\section{Introduction}
 Let $G$ be a connected semisimple real algebraic group and $\frak g=\Lie G$.
Let $A$ be a maximal real split torus of $G$ and set
$\mathfrak{a} = \Lie A$.
Fix a positive Weyl chamber $\fa^+<\fa$ and
  a maximal compact subgroup $K< G$ such that the Cartan decomposition $G=K (\exp \fa^+) K$ holds.
 We denote by $\mu(g)\in \fa$ the  Cartan projection of $g\in G$, that is, the unique element of $\fa^+$ such that $g\in K \exp(\mu(g))K$.
  Let $\Pi$ be the set of simple roots for $(\frak g, \frak a^+)$ and fix a non-empty subset $\theta\subset \Pi$. Let $P_\theta$ be the standard parabolic subgroup of $G$ corresponding to $\theta$ and set $$\F_\theta=G/P_\theta.$$

Let $\Ga<G$ be a Zariski dense discrete subgroup. 
Denote by $\La_\theta  \subset \F_{\theta}$ the limit set of $\Ga$, which is the unique $\Ga$-minimal subset of $\F_\theta$ \cite{Benoist1997proprietes}.
Let $\mathfrak{a}_\theta =\bigcap_{\alpha \in \Pi - \theta} \ker \alpha$.
For a linear form $\psi \in \fa_\theta^*$, a Borel probability measure $\nu$ on $\mathcal{F}_\theta$ is called a $(\Gamma, \psi)$-conformal measure if $$\frac{d \gamma_*\nu}{d\nu}(\xi)=e^{\psi(\beta_\xi^\theta(e,\gamma))} \quad \text{for all $\gamma \in \Gamma$ and $ \xi \in \mathcal{F}_\theta$} $$
 where $\ga_* \nu(B) = \nu(\ga^{-1}B)$ for any Borel subset $B\subset \F_\theta$ and $\beta_\xi^\theta$ denotes the $\fa_\theta$-valued Busemann map defined in \eqref{Bu}. By a $\Ga$-Patterson-Sullivan measure on $\F_\theta$, we mean 
 a $(\Ga,\psi)$-conformal measure supported on $\La_\theta$ for some $\psi \in \fa_{\theta}^*$. 
 
 Patterson-Sullivan measures play a fundamental role in the study of geometry and dynamics for $\Ga$-actions. For $G$ of rank one, they were  constructed by Patterson and Sullivan for any non-elementary discrete subgroup $\Ga$ of $G$ (\cite{Patterson1976limit}, \cite{Sullivan1979density}), and hence the name. Their construction was generalized by Quint for any Zariski dense subgroup of a semisimple real algebraic group \cite{Quint2002Mesures}.

A finitely generated subgroup $\Ga<G$ is called a $\theta$-Anosov subgroup if there exist $C_1, C_2>0$ such that for all $\ga\in \Ga$ and $\alpha\in \theta$,
$$\alpha(\mu(\ga))\ge C_1|\ga| -C_2$$  where $|\ga|$ denotes the word length of $\ga$ with respect to a fixed finite generating set of $\Ga$.
  A $\theta$-Anosov subgroup is necessarily a word hyperbolic group  \cite[Theorem 1.5, Corollary 1.6]{Kapovich2018morse}.
 The notion of Anosov subgroups was first introduced by Labourie for surface groups \cite{Labourie2006anosov}, and was extended to general word hyperbolic groups by Guichard-Wienhard \cite{Guichard2012anosov}.
Several equivalent characterizations have been  established, one of which is the above definition (see \cite{Gueritaud2017anosov} \cite{Kapovich2018discrete} \cite{Kapovich2017anosov} \cite{Kapovich2018morse}). Anosov subgroups are regarded as natural generalizations of convex cocompact subgroups of rank one groups, and include the images of Hitchin representations and of maximal representations as well as higher rank Schottky subgroups; see (\cite{Wienhard_ICM}, \cite{Kassel_ICM}).

A special case of our main theorem is the following non-concentration property of Patterson-Sullivan measures for $\theta$-Anosov subgroups:
\begin{theorem} \label{main.Anosov}
    Let $\Ga<G$ be a Zariski dense $\theta$-Anosov subgroup. For any
    $\Ga$-Patterson-Sullivan measure $\nu$ on $\F_\theta$, we have
    $$\nu(S)=0$$
    for any proper subvariety $S$ of $ \F_\theta$.
\end{theorem}

\begin{Rmk}\rm 
This  was proved by Flaminio-Spatzier \cite{Flaminio1990geometrically} for $G=\SO(n,1)$, $n\ge 2$, and
by Edwards-Lee-Oh  \cite{edwards2022torus} when $\theta=\Pi$ and the opposition involution of $G$ is trivial \eqref{eqn.opposition}.
\end{Rmk}

Indeed,  we work with a more general class of discrete subgroups, called $\theta$-transverse subgroups. Denote by $\i$ the opposition involution of $G$ (see \eqref{eqn.opposition}).
\begin{Def} \label{tra} A discrete subgroup $\Ga < G$ is called \emph{$\theta$-transverse} if
\begin{itemize}
    \item   it is \emph{$\theta$-regular}, i.e.,
$ \liminf_{\ga\in \Ga} \alpha(\mu({\ga}))=\infty $ for all $\alpha\in \theta$; and
\item it is  \emph{$\theta$-antipodal}, i.e.,
if any two distinct $\xi, \eta\in \La_{\theta\cup \i(\theta)}$ 
are in general position.
\end{itemize}
\end{Def}
Since $\i (\mu( g))=\mu(g^{-1}) $ for all $g\in G$, it follows that $\Ga$ is $\theta$-transverse if and only if $\Ga$ is $\i(\theta)$-transverse. The class of $\theta$-transverse subgroups includes all discrete subgroups of rank one Lie groups, $\theta$-Anosov subgroups and relatively $\theta$-Anosov subgroups.

Let $p_{\theta} : \fa \to \fa_{\theta}$ be the projection  which is invariant under all Weyl elements fixing $\fa_{\theta}$ pointwise, and set $\mu_\theta=p_\theta\circ \mu$.
A linear form $\psi \in \fa_{\theta}^*$ is said to be $(\Ga, \theta)$-proper if the composition $\psi \circ  \mu_\theta  : \Ga \to [-\varepsilon, \infty)$ is a proper map  for some $\varepsilon > 0$. 
The  following is our main theorem from which Theorem \ref{main.Anosov} is deduced by applying Selberg's lemma \cite{Selberg1960discontinuous}.
\begin{theorem} \label{main} 
    Let $\Ga<G$ be a Zariski dense virtually torsion-free $\theta$-transverse subgroup.  Let $\psi \in \fa_{\theta}^*$ be a $(\Ga, \theta)$-proper linear form  such that $\sum_{\ga \in \Ga} e^{-\psi(\mu_{\theta}(\ga))} = \infty$.
    For any  $(\Ga, \psi)$-Patterson-Sullivan measure $\nu$ on $\F_{\theta}$, we have
    $$\nu(S)=0$$
    for any proper subvariety $S$ of $ \F_\theta$.
\end{theorem}

 For a $\theta$-Anosov $\Ga$, the existence of a $(\Ga, \psi)$-Patterson-Sullivan measure implies that $\psi$ is $(\Ga, \theta)$-proper and $\sum_{\ga \in \Ga} e^{-\psi(\mu_{\theta}(\ga))} = \infty$ (\cite{lee2020invariant}, \cite{sambarino2022report}). Therefore Theorem \ref{main.Anosov} is a special case of Theorem \ref{main}.

\bigskip 
Added to the proof: 
The growth indicator $\psi_\Ga^\theta$ is a higher rank version of the classical critical exponent of $\Ga$ (\cite{Quint2002divergence}, \cite{kim2023growth}).
For a Zariski dense $\theta$-transverse subgroup and a $(\Ga, \theta)$-proper $\psi \in \fa_{\theta}^*$, the existence of a $(\Ga, \psi)$-conformal measure implies that $\psi$ is bounded from below by
$\psi_\Ga^\theta$  (\cite[Theorem 8.1]{Quint2002Mesures} for $\theta = \Pi$, \cite[Theorem 1.4]{kim2023growth} for a general $\theta$).
When $\Ga$ is relatively $\theta$-Anosov and $\psi$ is tangent to $\psi_\Ga^\theta$,  the abscissa of convergence of the series $s \mapsto \sum_{\ga \in \Ga} e^{-s\psi(\mu_{\theta}(\ga))}$ is equal to $1$, and a recent work \cite[Theorem 1.1]{CZZ_relative} shows that
 $\sum_{\ga \in \Ga} e^{-\psi(\mu_{\theta}(\ga))} = \infty$. Therefore 
Theorem \ref{main} also applies in this setting.

\subsection*{Acknowledgement} We would like to thank Subhadip Dey for helpful conversations.

\section{Ergodic properties of Patterson-Sullivan measures}\label{ergs}
Let $G$ be a connected semisimple real algebraic group.
Let $P < G$ be a minimal parabolic subgroup with a fixed Langlands decomposition $P=MAN, $ where $A$ is a maximal real split torus of $G$, $M$ is a maximal compact subgroup commuting with $A$, and $N$ the unipotent radical of $P$. We fix a positive Weyl chamber $\mathfrak{a}^+\subset \mathfrak{a} = \op{Lie}A$ so that $\log N$ consists of positive root subspaces. 
Recall that $K< G$ denotes a maximal compact subgroup such that the Cartan decomposition $G=K(\exp \fa^+) K$ holds and denote by $\mu:G\to \mathfrak{a}^+$ the Cartan projection, i.e., $\mu(g) \in \mathfrak{a}^+$ is the unique element such that $g\in K\exp(\mu(g))K$
 for $g\in G$.
Let $w_0\in K$ be an element of the normalizer of $A$ such that $\op{Ad}_{w_0}\mathfrak a^+= -\mathfrak a^+$.
  The opposition involution  $\i:\mathfrak a \to \mathfrak a$ is defined by
  \be \label{eqn.opposition}
  \i (u)= -\op{Ad}_{w_0} (u) \quad\text{for $u\in \fa$. }
  \ee
Note that $\mu(g^{-1})=\i (\mu(g))$ for all $g\in G$.

 Let $\Pi$ denote the set of all simple roots for $(\frak g, \frak a^+)$.
Fix a non-empty subset $\theta\subset \Pi$. Let $ P_\theta^-$ and $P_\theta^+$ be a pair of opposite standard parabolic subgroups of $G$ corresponding to $\theta$; here $P_\theta:=P_\theta^-$ is chosen to contain $P$. 
 We set $$\F_\theta^-=G/P_{\theta}^-\quad \text{ and } \quad \F_\theta^{+}=G/P_{\theta}^{+} .$$
We also write $\F_\theta=\F_\theta^-$ for simplicity.
We set $P=P_\Pi$ and $\F=\F_\Pi$. Since $P_\theta^+$ is conjugate to $P_{\i(\theta)}$,
    we have $\F_{\i(\theta)}=\F_\theta^+$.
   We say $\xi\in \F_\theta$ and $\eta\in \F_{\i(\theta)}$ are in general position if $(\xi, \eta)\in G (P_\theta^-, P_\theta^+)$ under the diagonal $G$-action on $\F_\theta\times \F_{\i(\theta)}$. We write
    $$\F_\theta^{(2)}= G (P_\theta^-, P_\theta^+),$$ which is the unique open $G$-orbit
    in $\F_\theta\times \F_{\i(\theta)}$.
    
Let $\mathfrak{a}_\theta =\bigcap_{\alpha \in \Pi - \theta} \ker \alpha$ and denote by $\fa_{\theta}^*$ the space of all linear forms on $\fa_{\theta}$. We set $p_\theta:\mathfrak{a}\to\mathfrak{a}_\theta$ the unique projection invariant under the subgroup of the Weyl group  fixing  $\fa_\theta$ pointwise. Set $\mu_{\theta} := p_{\theta} \circ \mu$.

The $\frak a$-valued Busemann map $\beta: \cal F\times G \times G \to\frak a $ is defined as follows: for $\xi\in \cal F$ and $g, h\in G$,
$$  \beta_\xi ( g, h):=\sigma (g^{-1}, \xi)-\sigma(h^{-1}, \xi)$$
where  $\sigma(g^{-1},\xi)\in \fa$
is the unique element such that $g^{-1}k \in K \exp (\sigma(g^{-1}, \xi)) N$ for any $k\in K$ with $\xi=kP$.
For $\xi=kP_\theta\in \cal F_\theta$ for $k\in K$, we define the $\fa_{\theta}$-valued Busemann map $\beta^{\theta} : \F_{\theta} \times G \times G \to \fa_{\theta}$ as
 \be\label{Bu} \beta_{\xi}^\theta (g, h): = 
p_\theta ( \beta_{kP} (g, h))\in \fa_\theta;\ee 
this is well-defined \cite[Section 6]{Quint2002Mesures}.

In the rest of this section, let  $\Ga < G$ be a Zariski dense $\theta$-transverse subgroup as in Definition \ref{tra}. For a $(\Ga, \theta)$-proper linear form $\psi\in \fa_\theta^*$,
we denote by $\delta_\psi\in (0, \infty]$ the abscissa of convergence of the series $\cal P_{\psi}(s) := \sum_{\ga\in \Ga}e^{-s \psi(\mu_{\theta}(\ga))}$; this is well-defined \cite[Lemma 4.2]{kim2023growth}. We set
$$\cal D_{\Ga}^{\theta}:=\{\psi\in \fa_\theta^*: (\Ga, \theta)\text{-proper, } \delta_\psi=1 \text{ and } \cal P_{\psi}(1)= \infty\}.$$ 
Note that $\psi \circ \i$ can be regarded as a linear form on $\fa_{\i(\theta)}$. Using the property that $\i (\mu(g))=\mu(g^{-1})$ for all $g\in G$,
we deduce that $\cal P_{\psi} = \cal P_{\psi \circ \i}$ and hence $\psi \in \cal D_{\Ga}^{\theta}$ if and only if $\psi \circ \i \in \cal D_{\Ga}^{\i(\theta)}$.

 The $\theta$-limit set $\La_{\theta}$ of $\Ga$ is the unique $\Ga$-minimal subset of $\F_{\theta}$ \cite{Benoist1997proprietes}. We also write 
    \be\label{ge} \La_\theta^{(2)} := \{(\xi, \eta)\in \F_\theta^{(2)}:\xi\in \La_\theta,\eta\in  \La_{\i(\theta)}\}.\ee 
The following ergodic property of Patterson-Sullivan measures was obtained by Canary-Zhang-Zimmer \cite{Canary2023} (see also \cite{kim2023growth}, \cite{kim2023ergodic}).

\begin{theorem}
\cite[Proposition 9.1, Corollary 11.1]{Canary2023} \label{thm.czz} \label{na} Suppose that $\theta=\i(\theta)$. Let $\Ga<G$ be a Zariski dense $\theta$-transverse subgroup.
For any $\psi\in \cal D_{\Ga}^{\theta}$, there exists a unique $(\Ga, \psi)$-Patterson-Sullivan measure $\nu_\psi$ on $\La_{\theta}$
and $\nu_\psi$ is  non-atomic.
Moreover, the diagonal $\Ga$-action on $(\La_{\theta}^{(2)}, (\nu_\psi \times \nu_{\psi\circ \i})|_{\La_{\theta}^{(2)}})$ is ergodic.
    \end{theorem}

\section{A property of convergence group actions}\label{convs}
In this section, we prove a certain property of convergence group actions which we will need in the proof of our main theorem in the next section. We refer to \cite{Bowditch1999convergence} for basic properties of convergence group actions.  Let $\Ga$ be a countable group acting on a compact metrizable space $X$  (with $\# X\ge 3$) by homeomorphisms. 
This action is called a \emph{convergence group action} if for any sequence of distinct elements $\ga_n \in \Ga$, there exist  a subsequence $\ga_{n_k}$ and $a, b \in X$ such that as $k \to \infty$, 
$\ga_{n_k}(x) $ converges to $ a $ for all $x\in X-\{b\}$, uniformly on compact subsets. 
In this case, we say $\Ga$ acts on $X$ as a \emph{convergence group}, which we suppose in the following.
Any element $\ga\in \Ga$ of infinite order
fixes precisely one or two points of $X$, and  $\ga$ is called parabolic or loxodromic accordingly.
In that case, there exist  $a_{\ga}, b_{\ga} \in X$, fixed by $\ga$, such that $\ga^n|_{X- \{b_{\ga}\}} \to a_\ga$ uniformly on compact subsets as $n \to \infty$. We have
 $\ga$ loxodromic if and only if $a_{\ga}\ne b_{\ga}$, in which case
 $a_{\ga}$ and $b_\ga$ are called the attracting and repelling fixed points of $\ga$ respectively.

We will use the following lemma in the next section:

\begin{lemma} \label{lem.convact}\label{cont}
    Let $\Ga$ be a torsion-free countable group acting on a compact metric space $X$ as a convergence group. For any compact subset $W$ of $X$ with at least two points, the subgroup $\Ga_W = \{\ga \in \Ga : \ga W = W\}$ acts on $X- W$ properly discontinuously, that is, for any $\eta \in X - W$, there exists an open neighborhood $U$ of $\eta$ such that $\ga U \cap U \neq \emptyset$ for $\ga \in \Ga_W$ implies $\ga = e$. 
\end{lemma}

\begin{proof}
Suppose not. Then there exist $\eta \in X - W$,
a decreasing sequence of open neighborhoods $U_n$ of $\eta$ in $X$ with $\bigcap_n U_n=\{\eta\}$ and a sequence $e\ne \ga_n\in \Ga$ such that  $\ga_n W=W$ and $\ga_n U_n \cap U_n \neq \emptyset$ 
for each $n\in \N$. Hence there exists a sequence $\eta_n \in U_n\cap \ga_n^{-1}U_n$; so $\eta_n \to \eta$ and $\ga_n \eta_n \to \eta$ as $n \to \infty$.

We claim that the elements $\ga_n$ are all pairwise distinct, possibly after passing to a subsequence. Otherwise, it would mean that, after passing to a subsequence, $\ga_n$'s are constant sequence, say $\ga_n = \ga\ne e$. Since $\ga \eta = \lim_n \ga_n \eta_n = \eta$,  $\eta$ must be a fixed point of $\ga$. Since $\Ga$ is torsion-free, $\ga$ is either parabolic or loxodromic, and in particular it has at most two fixed points in $X$, including $\eta$. Since $\eta \not\in W$ and $W$ has at least two points, we can take $w \in W$ which is not fixed by $\ga$. Then as $n\to +\infty$, $\ga^n w\to \eta$ or $\ga^{-n}w\to \eta$. 
Since $W$ is a compact subset such that $\ga W=W$ and $\eta \notin W$, this yields a contradiction.

Therefore we may assume that $\{\ga_n\}$ is an infinite sequence of distinct elements. Since the action of $\Ga$ on $X$ is a convergence group action, there exist a subsequence $\ga_{n_k}$  and $a, b \in X$ such that as $k\to \infty$,
$\ga_{n_k}(x)$ converges to  $ a$  for all $x\in X - \{b\}$, uniformly on compact subsets. There are two cases to consider.
Suppose that $b=\eta$.
Then  $W \subset X - \{b\}$, and hence  $\ga_{n_k}W \to a$ uniformly as $k \to \infty$. Since $\ga_{n_k} W = W$ and $W$ is a compact subset, it follows that $W=\{a\}$, contradicting the hypothesis that $W$  consists of at least two elements. 
Now suppose that $b \neq \eta$.
Since $\eta_{n_k}$ converges to $ \eta$, we may assume that $\eta_{n_k} \neq b$ for all $k$. Noting that $\# W \ge 2$, we can take $w_0 \in W - \{b\}$. If we now consider the following compact subset $$W_0 := \{\eta_{n_k} : k \in \N \} \cup \{\eta, w_0\} \subset X - \{b\},$$
we then have $\ga_{n_k} W_0 \to a$ uniformly as $k \to \infty$. Since $\eta_{n_k} \in W_0$ for each $k$ and $\ga_{n_k}\eta_{n_k} \to \eta$ as $k \to \infty$, we must have $$a = \eta.$$

On the other hand, since $w_0 \in W_0\cap W$, $\ga_{n_k} w_0 \to \eta$ as $k \to \infty$. This implies $\eta\in W$ since $W$ is compact and $\ga_{n_k}w_0\in W$, yielding a contradiction to the hypothesis $\eta\notin W$. This completes the proof.
\end{proof}

We denote by $\La_X$ the set of all accumulation points of a $\Ga$-orbit in $X$.  If $\#\La_X\ge 3$,
the $\Ga$-action is called non-elementary and $\La_X$ is the unique $\Ga$-minimal subset \cite{Bowditch1999convergence}.

A well-known example of a convergence group action is given by a word hyperbolic group $\Ga$. Fix a finite symmetric generating subset $S_{\Ga}$ of $\Ga$.  A geodesic ray in $\Ga$ is an infinite  sequence $(\ga_i)_{i=0}^{\infty}$ of elements of $\Ga$ such that $\ga_i^{-1}\ga_{i+1}\in S_{\Ga}$ for all $i\ge 0$. The Gromov boundary $\partial\Ga$ is the set of equivalence classes of geodesic rays, where two rays are equivalent to each other if and only if their Hausdorff distance is finite.
The group $\Ga$ acts on $\partial \Ga$ by  $\ga \cdot [(\ga_i)]=[(\ga \ga_i)]$. 
This action is known to be a convergence group action \cite[Lemma 1.11]{Bowditch1999convergence}.

Another important example of a convergence group action is the action of a $\theta$-transverse subgroup $\Ga$ on $\La_{\theta \cup \i(\theta)}$:
\begin{prop} \cite[Theorem 4.21]{Kapovich2017anosov} \label{prop.transisconv}
For a $\theta$-transverse subgroup $\Ga$,
the action of  $\Ga$ on $\La_{\theta \cup \i(\theta)}$ is a convergence group action.
\end{prop}

\section{Non-concentration property}
We fix a non-empty subset $\theta \subset \Pi$. We first prove the following proposition from which we will deduce Theorem \ref{main}.

\begin{proposition} \label{prop.torsionfree}\label{tf}
 Let $\Ga < G$ be a torsion-free Zariski dense
 discrete subgroup admitting a convergence group action on a compact metrizable space $X$. We assume that this action is $\theta$-antipodal in the sense that  there exist $\Ga$-equivariant homeomorphisms $f_{\theta} : \La_X \to \La_{\theta}$ and $f_{\i(\theta)} : \La_X \to \La_{\i(\theta)}$  such that  for any $\xi \ne \eta$ in $\La_X$,
 $$\left(f_{\theta}(\xi), f_{\i(\theta)}(\eta)\right)\in \La_\theta^{(2)}.$$ 
    Let $\nu$ be a $\Ga$-quasi-invariant measure on $\La_{\theta}$ such that 
    \begin{enumerate}
        \item $\nu$ is non-atomic;
\item  $\Ga$ acts ergodically on $(\La_{\theta}^{(2)}, 
    (\nu \times \nu_{\i})|_{\La_\theta^{(2)}})$ for some $\Ga$-quasi-invariant measure $\nu_{\i}$ on $\La_{\i(\theta)}$.
       \end{enumerate}
       Then for any proper algebraic subset $S$ of $\F_{\theta}$, we have $$\nu(S) = 0.$$
\end{proposition}

\begin{proof} We first claim that
the $\Ga$-action on $(\La_\theta \times \La_{\i(\theta)}, \nu \times \nu_{{\i}})$ is ergodic.   
Set
$R:=(\La_\theta \times\La_{\i(\theta)} ) - \La_\theta^{(2)}$.
Since  the $\Ga$-action on $(\La_{\theta}^{(2)}, 
    (\nu \times \nu_{\i})|_{\La_\theta^{(2)}})$ is ergodic,
it suffices to show that
 $$(\nu\times\nu_{{\i}})(R)=0.$$
For $y\in \La_{\i(\theta)}$, let
$R(y):=\{x\in\La_{\theta} : (x,y) \in R\}$.
By the antipodal property of the pair $(f_\theta, f_{\i(\theta)})$, we have that  for each $y\in\La_{\i(\theta)}$, 
we have $R(y)=\emptyset$ or $R(y)= \{(f_{\theta} \circ f_{\i(\theta)}^{-1})(y)\}$ and hence $\nu( R(y))=0$ 
by the non-atomicity of $\nu$. 

Therefore
\begin{equation}\label{eq.B}
(\nu\times\nu_{{\i}})(R)=\int_{y\in \La_{\i(\theta)}}\nu (R(y))\,d\nu_{\i}(y)=0,
\end{equation}
proving the claim.

Now suppose that $\nu(S) > 0$ for some proper algebraic subset $S \subset \F_{\theta}$. We may assume that $S$ is irreducible and
of minimal dimension among all such algebraic subsets of $\F_\theta$.
Let $W = f_{\theta}^{-1}(S\cap \La_\theta) \subset \La_X$. Since $\nu$ is non-atomic and $\nu(S) > 0$, we have $\# W = \infty > 2$.
This implies $\#\La_X\ge 3$. By the property of a non-elementary convergence group action, 
$\La_X$ is the unique $\Ga$-minimal subset of $X$ and there always exists a loxodromic element of $\Ga$ \cite{Bowditch1999convergence}.

Since $\Ga < G$ is Zariski dense, $\La_\theta$ is Zariski dense in $\F_\theta$ as well, and hence $\La_{\theta} \not\subset S$. Therefore $X - W$ is a non-empty open subset intersecting $\La_X$. Since $\Ga$ acts minimally on $\La_X$ and the set of attracting fixed points of loxodromic elements of $\Ga$ is a non-empty $\Ga$-invariant subset,
there exists a loxodromic element $\ga_0 \in \Ga$ whose attracting fixed point $a_{\ga_0}$ is contained in $\La_X- W$.
    Hence applying Lemma \ref{lem.convact} to $\eta=a_{\ga_0}$, we have an open neighborhood $U $ of $a_{\ga_0}$ in $\La_X$ such that 
    \be \label{eqn.Uchoice}
    \ga U \cap U = \emptyset
    \ee for all non-trivial $\ga \in \Ga$ with $\ga W = W$.

    Since $\ga_0^m|_{\La_X - \{b_{\ga_0} \}} \to a_{\ga_0}$ uniformly on compact subsets as $m \to  +\infty$ and $\#\La_X\ge 3$,
    $U$ contains a point $\xi\in \La_X-\{a_{\ga_0}, b_{\ga_0}\}$. By replacing $\ga_0$ by a large power $\ga_0^m$ if necessary,
    we can find an open neighborhood $V$ of $\xi$ contained in $U-\{a_{\ga_0}\}$
   such that $\ga_0 V \subset U$ and $\ga_0 V \cap V = \emptyset$.
   
    We now consider the subset $$S \times f_{\i(\theta)}(V)$$ of $\F_{\theta} \times \F_{\i(\theta)}$. Since $\nu(S) > 0$ and $\nu_{\i}(f_{\i(\theta)}(V)) > 0$, we have that $\Ga(S \times f_{\i(\theta)}(V))$ has full $\nu \times \nu_{\i}$-measure by the ergodicity of the $\Ga $-action on $(\La_{\theta} \times \La_{\i(\theta)}, \nu \times \nu_{\i})$.
    Since $(\nu \times \nu_{\i})(S \times \ga_0 f_{\i(\theta)}(V)) > 0$, there exists $\ga \in \Ga$ such that $$(\nu \times \nu_{\i})\left( (S \times \ga_0 f_{\i(\theta)}(V)) \cap (\ga S \times \ga f_{\i(\theta)}(V)\right) > 0.$$
    In particular, we have $$\nu(S \cap \ga S) > 0 \quad \mbox{and} \quad \nu_{\i}(\ga_0 f_{\i(\theta)}(V) \cap \ga f_{\i(\theta)}(V)) > 0.$$
    Since $S$ was chosen to be of minimal dimension and irreducible among proper algebraic sets with positive $\nu$-measure, we must have $S = \ga S$. It follows from the $\Ga$-invariance of $\La_{\theta}$ that $W = \ga W$.

    The $\Ga$-equivariance of $f_{\i(\theta)}$ implies that 
    \be \label{eqn.posmeas}
    \nu_{\i}(f_{\i(\theta)}(\ga_0 V \cap \ga V)) > 0.
    \ee
    Since $\ga_0 V \cap V = \emptyset$, we have $\ga \neq e$. Hence it follows from $V \subset U$, $\ga_0 V \subset U$ and the choice \eqref{eqn.Uchoice} of $U$ that $$\ga_0 V \cap \ga V \subset U \cap \ga U = \emptyset,$$ which gives a contradiction to \eqref{eqn.posmeas}. This finishes the proof.
\end{proof}

\subsection*{Proof of Theorem \ref{main}}
Let $\Ga < G$ be a Zariski dense $\theta$-transverse subgroup and $\nu$ a $(\Ga, \psi)$-Patterson-Sullivan measure for a $(\Ga, \theta)$-proper linear form $\psi \in \fa_{\theta}^*$ such that $\sum_{\ga \in \Ga} e^{-\psi(\mu_{\theta}(\ga))} = \infty$. We may assume without loss of generality that $\Ga$ is torsion-free. Indeed, let $\Ga_0 < \Ga$ be a torsion-free subgroup of finite index. Then $\Ga_0$ is also a Zariski dense $\theta$-transverse subgroup of $G$. Moreover, $\nu$ is a $(\Ga_0, \psi)$-Patterson-Sullivan measure since the limit sets for $\Ga$ and $\Ga_0$ are same. Write $\Ga = \bigcup_{i = 1}^n \ga_i \Ga_0$ for some $\ga_1, \cdots, \ga_n \in \Ga$. By \cite[Lemma 4.6]{Benoist1997proprietes}, there exists $C > 0$ such that $\| \mu(\ga_i \ga) - \mu(\ga)\| \le  C$ for all $\ga \in \Ga_0$ and $i = 1, \cdots n$. Hence we have that $\psi$ is $(\Ga_0, \theta)$-proper as well and $$\infty = \sum_{\ga \in \Ga} e^{-\psi(\mu_{\theta}(\ga))} = \sum_{i = 1}^n \sum_{\ga \in \Ga_0} e^{-\psi(\mu_{\theta}(\ga_i \ga))} \le  n e^{\|\psi\| C} \sum_{\ga \in \Ga_0} e^{-\psi(\mu_{\theta}(\ga))}$$
where $\|\psi\|$ denotes the operator norm of $\psi$.
In particular, $\sum_{\ga \in \Ga_0} e^{-\psi(\mu_{\theta}(\ga))} = \infty$. Therefore, replacing $\Ga$ by $\Ga_0$, we assume that $\Ga$ is torsion-free. By Proposition \ref{prop.transisconv}, the action of $\Ga$ on $\La_{\theta \cup \i(\theta)}$ is a convergence group action.

Since there exists  a $(\Ga, \psi)$-conformal measure, we have $\delta_{\psi} \le 1$ by \cite[Lemma 7.3]{kim2023growth}. Therefore the hypothesis $\sum_{\ga \in \Ga} e^{-\psi(\mu_{\theta}(\ga))}  = \infty$ implies that $\psi \in \cal D_{\Ga}^{\theta}$. Moreover, the $\theta$-antipodality of $\Gamma$ implies that
the canonical projections $$f_{\theta} : \La_{\theta \cup \i(\theta)} \to \La_{\theta} \quad \text{and} \quad f_{\i(\theta)} : \La_{\theta \cup \i(\theta)} \to \La_{\i(\theta)}$$ are $\Ga$-equivariant $\theta$-antipodal  homeomorphisms \cite[Lemma 9.5]{kim2023growth}.
This implies that  Theorem \ref{na} indeed holds for a general $\theta$ without the hypothesis $\theta=\i(\theta)$.
Hence $\nu=\nu_\psi$, $\nu_\psi$ is non-atomic and the diagonal $\Ga$-action on $(\La_{\theta}^{(2)}, (\nu_\psi \times \nu_{\psi\circ \i})|_{\La_{\theta}^{(2)}})$ is ergodic. Since $\nu_{\psi \circ \i}$ is $\Ga$-conformal, it is $\Ga$-quasi-invariant. Therefore Theorem \ref{main} follows from Proposition \ref{prop.torsionfree}.

\medskip

We emphasize again that Lemma \ref{cont} and Proposition \ref{tf} were
introduced to deal with the case when $\i$ is non-trivial.
Indeed, when $\i$ is trivial,  Theorem
\ref{main} follows from the following $\theta$-version of \cite[Theorem 9.3]{edwards2022torus}.
\begin{theorem} \label{final}
    Let $\Ga < G$ be a Zariski dense discrete subgroup. Let $\nu$ be a $\Ga$-quasi-invariant measure on $\La_{\theta}$. Suppose that the diagonal $\Ga$-action on $(\La_{\theta} \times \La_{\theta}, \nu \times \nu)$ is ergodic. Then for any proper algebraic subset $S$ of $\F_{\theta}$, we have $$\nu(S) = 0.$$
\end{theorem}
\begin{proof} The proof is identical to the proof of \cite[Theorem 9.3]{edwards2022torus} except that we work with a general $\theta$. We reproduce it here for the convenience of readers.
Let $S$ be a proper irreducible subvariety of $\F_\theta$ with $\nu(S)>0$
and of minimal dimension.  Since $(\nu\times \nu) (S\times  S)>0$,  the $\G$-ergodicity of $\nu\times \nu$ implies that
 $(\nu\times \nu)(\Ga(S\times S))=1$. It follows that for any $\ga_0\in \Ga$,
 there  exists $\ga\in \Ga$ such that
 $(S \times \ga_0S)\cap (\ga S\times \ga S)$ has positive $\nu\times \nu$-measure; hence $\nu(S\cap \ga S)>0$ and
$\nu(\ga_0 S\cap \ga S)>0$. Since $S$ is irreducible and of minimal dimension, it follows that $S=\ga S=\ga_0 S$.
Since $\ga_0\in \Ga$ was arbitrary, we have $\Gamma S=S$, contradicting the Zariski density hypothesis on $\Ga$.
\end{proof}

We finally mention  that the proof of Proposition \ref{prop.torsionfree}
implies the following when the second measure cannot be taken to the same as the first measure:
\begin{theorem} 
    Let $\Ga < G$ be a Zariski dense torsion-free discrete subgroup acting on $\La_\theta$ as a convergence
    group. Let $\nu$ be a non-atomic $\Ga$-quasi-invariant measure on $\La_{\theta}$. Suppose that the diagonal $\Ga$-action on $(\La_{\theta} \times \La_{\theta}, \nu \times \nu')$ is ergodic for some $\Ga$-quasi-invariant measure $\nu'$ on $\La_\theta$. Then for any proper algebraic subset $S$ of $\F_{\theta}$, we have $$\nu(S) = 0.$$
\end{theorem}

\begin{proof}
    Since $\Ga$ acts ergodically on the entire {\it product space} $(\La_{\theta} \times \La_{\theta}, \nu \times \nu')$, the first part of the proof of Proposition \ref{prop.torsionfree} is not relevant.
Suppose that $S$ is an irreducible proper subvariety of $\cal F_\theta$ and
of minimal dimension among all subvarieties with positive $\nu$-measure. Then setting $W = S \cap \La_{\theta}$, 
 as in the proof of  Proposition \ref{prop.torsionfree},
 we can find non-empty open subsets $V\subset U\subset \La_\theta- W$ such that $
    \ga U \cap U = \emptyset$ for all non-trivial $\ga \in \Ga$ with $\ga W = W$,
    and $\ga_0V\subset U$ and $\ga_0V\cap V=\emptyset$
    for some $\ga_0\in \Ga$.
Using $(\nu\times \nu')(S\times V)>0$, we then get a contradiction by the same argument in loc. cit.
\end{proof}

\bibliographystyle{plain} 

\end{document}